\documentclass{amsart}

\newtheorem{theorem}{Theorem}[section]
\newtheorem{lemma}[theorem]{Lemma}
\newtheorem{proposition}[theorem]{Proposition}
\newtheorem{corollary}[theorem]{Corollary}

\theoremstyle{definition}
\newtheorem{definition}[theorem]{Definition}
\newtheorem{example}[theorem]{Example}

\theoremstyle{remark}
\newtheorem{remark}[theorem]{Remark}

\numberwithin{equation}{section}

\usepackage{amsmath,amsfonts,amssymb}

\begin{document}

\title{On a new general method of  summation}

\author{Armen G. Bagdasaryan}
\address{Russian Academy of Sciences, 65 Profsoyuznaya, 117997 Moscow, Russia
              }
\email{bagdasar@member.ams.org}



\date{October 25, 2011} 



\begin{abstract}
A new general and unified method of summation, which is both regular and consistent, is invented. The generality and unification come from the intrinsic nature of the method, which is based on the fundamental idea concerning the method for ordering the integers, and the definition of sum that generalizes and extends the
usual one to the case when the upper limit of summation is less than the lower. The resulting theory includes a number of explicit and closed form summation formulas, and assigns limits to certain unbounded or oscillating functions. Some problems
and future lines of research within this more general theoretical setting are briefly discussed.
\end{abstract}

\maketitle

\tableofcontents


\section{Introduction}

Every summation method is a method of assigning numerical value to series consisting of infinite number of terms, called sum of the series.
To justify the term ``summation'', the connection between the terms of series and its ``sum'' should possess that kind of connection which takes place between the finite series and their usual sums. One definition that satisfies this requirement is a classical definition that defines the sum of infinite series as a finite limit of the sequence of its partial sums \cite{knopp}.

In the case of finite series, one seeks a closed-form expression for sums of the form $\sum_{u=a}^{b}f(u)$. As is known from finite calculus \cite{gelfond}, this problem is equivalent to solving the difference equation $g(u+1)-g(u)=f(u)$, and if there exists a solution $g(u)$ then we obtain the result $F(a,b)=g(b+1)-g(a)$.
In some cases we are able to find the analytical expression $F(a,b)$ for the sum  $\sum_{u=a}^{b}f(u), \, a\leq b$. 

However, sum of the form $\sum_{u=a}^{b}f(u)$ is not defined for $b<a$, not to mention $b<0$ -- negative number of terms, and by convention it is understood to be equal to zero.
In other words, sums $\sum_{a}^{b}$ are defined classically only when the number of terms is a positive integer or infinity. 

It should be noted that Euler derived the formula $\sum_{u=1}^{-1/2}\frac{1}{u}=-2\ln2$ that can be found in \cite{euler1}, in which the upper limit of summation is negative.  
So, one can suppose that Euler might have had similar ideas when calculating the values of infinite sums  \cite{euler1,euler2,euler3,euler4}.
Nevertheless, there have been no general works on  systematic studies to  elaborate this direction, apart from rare examples in Euler's works. 
But it was Euler  who started to systematically work with divergent series. 
In fact, Euler was convinced that ``to every series one could assign a number'' \cite{vara}, in a reasonable, consistent, and useful way, of course. Euler was unable to prove this statement in full and failed to give a rigorous foundations for divergent series, but he devised a technique (Euler's summation criterion) in order to sum large families of divergent series. 

For years, there was the suspicion that one could try to give real sense to divergent series.
This resulted in  a number of  methods that have been proposed to sum divergent series, which are due to Abel, Euler, Ces\`aro, Bernoulli, Dirichlet, Borel, Riesz, Ramanujan and some other mathematicians. The most powerful of them involve analytic continuation in the complex plane.

However, as is known, modern analysis fails in finding the limits of unbounded and oscillating functions and sequences, which is important in series summations, particularly if considering the remainder term, and in deriving a number of properties for divergent series and their summation from the unified and general standpoints of foundational nature. 

Our initial goal is to make sense of sums $\sum_{u=a}^{b}f(u)$ for $b<a$.  We define a sum in such a way that it retains its analytical expression, making the method regular, thus extending the traditional definition to the case when $b<a$, and providing more general setting for summation.

The main idea that we develop to define sums for arbitrary $a,b,\: a^>_<b$, including $b<0$, is to introduce a new ordering relation on the set of integer numbers. In addition, along with the usual definition of sum of a series, we introduce some additional conditions  -- ``axioms''; the model of this ``axiomatic system'' is  arithmetic.
By means of these, we define a regular summation method which can be applied to convergent series and many divergent series as well.  Apart from that, our method allows us to find the limits of certain unbounded and oscillating functions and sequences.
 
So, the purpose of this article is to present a new general and unified method for summation that provides a systematic way to extend summations to arbitrary limits $a$ and $b$.
The method makes it possible to get closed form evaluations, thus resulting for most of the cases in exact values and explicit formulas for the sums of infinite series. It also enables one to find finite and closed form summation formulas for various types of series.
In general, the theoretical setting developed in this paper leads to a number of nice and interesting results.

Throughout the paper, we present some examples which are supposed to give our readers a general understanding of how the method elaborated here actually works.
Most of them is concerned with the evaluation of Riemann's zeta and related functions at integer points.
These examples are usually simple, but the goal of these examples is to explain the general theory.
Of course, it should be mentioned that the sums for many of the series considered in our examples are known in the sense that they can be found in the literature, for example, in the book by Hardy \cite{hardy} or Titchmarsh \cite{titch}, or some others \cite{grad,joll}. 
However, the explicit form of our sum formulas and the most of general summation formulas have not been known so far.

The general structure of this article is as follows. In Section \ref{sec2} we define a new ordering relation on the set of integers, introduce a class of regular functions, give the definition of sum,  and define the summation method. Section \ref{sec:sums} deals with the properties of summation over newly ordered number line; 
Section \ref{sec:main} is the main part of our paper where we present our main results, including a number of summation formulas, several properties of divergent series,
and some propositions concerning the limits of unbounded and oscillating functions.
In Section \ref{sec:exten} we show how the class of regular functions can be substantially extended.
We conclude in Section \ref{sec:conc} with discussion of the method and obtained results, and with some aspects of future research
which will be based on the theory presented here.

\section{The method}\label{sec2}

Consider the set of all integer numbers $\mathbb{Z}$. We introduce a new ordering relation on the set $\mathbb{Z}$ as follows.

\begin{definition} \label{def:order}
We shall say that $a$ precedes $b$, $a, b \in \mathbb{Z}$, and write $a\prec b$, if the inequality $\dfrac{-1}{a}<\dfrac{-1}{b}$ holds; $a\prec b \Leftrightarrow \dfrac{-1}{a}<\dfrac{-1}{b}$. 
\end{definition}

In the definition \ref{def:order} we assume by convention that $0^{-1}=\infty$.

Alternatively, using the set theoretic language, this ordering can be defined in the following manner.

\begin{definition}
$a\prec b$ if and only if 
\begin{enumerate}
\item $a=0$
\item $a$ is positive, $b$ is negative
\item $a$ and $b$ both positive and $a<b$
\item $a$ and $b$ both negative and $|b|<|a|$
\end{enumerate}
\end{definition}

The introduced ordering relation can be visualized as follows 
$$
0 \quad 1 \quad 2 \quad 3 \: \dots \: n \quad \dots \quad -n \: \dots \: -3 \quad -2 \quad -1
$$

From this method of ordering it follows that any positive integer number, including zero, precedes any negative integer number, and the set $\mathbb{Z}$ has zero as the first element and $-1$ as the last element, that is, we have that the set $\mathbb{Z}=[0, 1, 2,...-2, -1]$\footnote{the set $\mathbb{Z}$ can be
imagined as  homeomorphic to a circle}. 
For the re-ordered set $\mathbb{Z}$ the following two essential axioms of order hold:

\begin{itemize}
\item Transitivity: \newline 
				if $a\prec b$ and $b\prec c$ then $a\prec c$ 
\item Totality: \newline 
				if $a\neq b$ then either $a\prec b$ or $b\prec a$ 
\end{itemize}

So, the introduced ordering $\prec$ defines a strict linear order on $\mathbb{Z}$.

\begin{remark}
Let us note that Euler's works \cite{euler2,euler3,euler4} show that he possibly had advanced the same idea through reasoning about infinite series. In his fundamental work \cite{euler2} Euler concluded that ``the sum $-1$ is larger than infinity'' \cite{dun,kline,klinebook,sand}, that is numbers beyond infinity might be negative.
We find from several sources that he ``argued that infinity separates positive and negative numbers just as $0$ does'' \cite{dun,kline,sand}. And in \cite{euler4} Euler 
``showed that positive and negative numbers are connected by crossing through infinity'' \cite{wiel}.
Although no one knows Euler's true idea, we think that discussion of this matter, having a little bit speculative character, is still of much interest, not only of historical value, so we include this remark for the sake of completeness of this study.
\end{remark}

Let $a,b$ be any integer numbers. Suppose $\mathbb{Z}_{a, b}$ is a \textit{part} of $\mathbb{Z}$ such that 
$$ \mathbb{Z}_{a, b} =               
									\begin{cases}                   
									[a, b]                     & \quad (a\preceq b)\\                   
									\mathbb{Z}\setminus (b, a) & \quad (a\succ b)              
									\end{cases}       						 
$$
where $\mathbb{Z}\setminus (b, a)=[a, -1]\cup[0, b]$ \footnote{the segment $[a, b]$ (interval $(a, b)$), as usual, is the set of all integers $x$ such that $a\preceq x\preceq b$ ($a\prec x\prec b$)}. 

Let $f(x)$ be a function of real variable defined on $\mathbb{Z}$. We introduce the following fundamental definition of sum.

\begin{definition} \label{def:sum}
For any $a, b\in \mathbb{Z}$
\begin{equation} 
\sum_{u=a}^b{f(u)}=\sum_{u\in \mathbb{Z}_{a, b}}{f(u)}.\footnote{observing the order of elements in $\mathbb{Z}_{a, b}$} \label{def:first_sum_eq}
\end{equation}
\end{definition}

This definition  satisfies the condition of generality and has a concrete sense for any integer values of $a$ and $b$ ($a ^{>}_{<} b$).
The definition \ref{def:sum} generalizes the usual definition of sum for the case $b<a$. 
The sum $\sum_{a}^{b}f(u)$ is defined for arbitrary limits of summation ($a ^{>}_{<} b$) and in such a way 
that its functional dependence for $a\leq b$ retains its analytical expression, i.e. it is regular.
The set $\mathbb{Z}_{a, b}$, depending on the elements $a$ and $b$, can be either finite or infinite. Thus, the sum on the right-hand side of (\ref{def:first_sum_eq}) can become an infinite series, which has to be given a numeric meaning.


In this study we restrict ourselves with a class of functions that we refer to as regular (but, we have to mention that quite a few of statements of our summation theory do not require the regularity of functions).
The studying of the class of regular functions is motivated by the wide applications of the theory of summation of functions, developed within the calculus of finite differences. 
The following definition introduces the notion of regular function.

\begin{definition}\label{def:reg}
The function $f(x)$, $x\in \mathbb{Z}$ , is called \textit{regular} if there exists an elementary\footnote{determined by formulas, constructed by a finite number of algebraic operations and constant functions and algebraic, trigonometric, and exponential
functions, and their inverses through repeated combinations and compositions} 
function $F(x)$ such that $F(x+1)-F(x)=f(x), \quad \forall \, x\in \mathbb{Z}$. The function $F(x)$ satisfying the above relation is said to be a \textit{primitive function} for $f(x)$.
\end{definition}

It should be noted that the primitive function is not unique for the given function $f(x)$. Namely, if $F(x)$ is a primitive function for $f(x)$, then the function $F(x)+C$, where $C$ is a constant, is also a primitive function for $f(x)$. Thus, any function $F(x)$, which is primitive for $f(x)$, can be represented in the form $F(x)+C(x)$, where $C(x)$ is a periodic function with the period $1$. 
The class of regular functions is large enough to be used in applications. 

From definition \ref{def:reg} we obtain

\begin{proposition}
Let $f_{i}(x)$, $i=1,2..,k, \, k\in\mathbb{N}$, be regular functions. Then the function $\varphi(x)=\sum_{i=1}^{k}\alpha_{i}f_{i}(x)=\alpha_{1}f_{1}(x)+\alpha_{2}f_{2}(x)+\dots+\alpha_{k}f_{k}(x), \; \alpha_i\in\mathbb{R}$, is also a regular function. That is, any linear combination of regular functions is a regular function.
\end{proposition}

We postulate that

\textit{every series $\sum_{u=1}^{\infty}f(u)$, where $f(u)$ is a regular function, has a certain finite numeric value}.\footnote{this assertion has something in common with the first Euler's principle concerning infinite series which states that ``to every series could be assigned a number'' \cite{vara}.} 


For the purpose of assigning a numerical value to (\ref{def:first_sum_eq}), 
we introduce several sufficiently general and natural axioms.

\begin{itemize}
\item[\bf A1.] \label{axiom:1} If \; $S_n=\sum_{u=a}^n{f(u)} \quad \forall \, n$, \ then \ $\lim_{n\rightarrow \infty}S_n=\sum_{u=a}^\infty {f(u)}$.       
\item[]
\item[\bf A2.] \label{axiom:2} If \; $S_n=\sum_{u=1}^{[n/2]}{f(u)} \quad \forall \, n$, \ then \  $\lim_{n\rightarrow \infty}S_n=\sum_{u=1}^\infty {f(u)}$.    
\item[]
\item[\bf A3.] \label{axiom:3} If \; $\sum_{u=a}^{\infty} f_1(u)=S_1$ \ and \ $\sum_{u=a}^{\infty}f_2(u)=S_2$, \ then \\ 
              $\sum_{u=a}^{\infty} \left(\alpha_1 f_1(u) + \alpha_2 f_2(u)\right)= \alpha_1 S_1+\alpha_2 S_2$. 
\item[]   
\item[\bf A4.] \label{axiom:4} If \ $G=[a_1, b_1]\cup[a_2, b_2]$, \ and \ $[a_1, b_1]\cap[a_2, b_2]=\emptyset$, \ then  $\sum_{u\in G}f(u)=\sum_{u=a_1}^{b_1}f(u)+\sum_{u=a_2}^{b_2}f(u)$.
\end{itemize}

The axioms \textbf{A1}--\textbf{A4} define the method of summation, which is regular due to axiom \textbf{A1}, 
since it sums every convergent series to its usual sum.   

It is worth to notice that introduced ``axioms'' are consistent\footnote{they do not contradict to
what has been obtained in the framework or on the basis of known means of analysis}
with statements and definitions of analysis: the axiom \textbf{A1} is the usual definition of sum of a series; the axiom \textbf{A2} also does not contradict to classical analysis; the axiom \textbf{A3}  constitutes the well-known property of convergent series\footnote{which have finite sums; bearing in mind our postulate, use of this axiom is in agreement with classical analysis};  the axiom \textbf{A4} has quite natural and common formal meaning.

Relying on axiom \textbf{A4}, one can make a couple of remarks on the sum (\ref{def:first_sum_eq}).

\begin{remark} \label{rem:a_bigger_b}
Let $b\prec a$. Then by definition $\mathbb{Z}_{a, b}=[a, -1]\cup[0, b]$ and $[a, -1]\cap[0, b]=\emptyset$. And in view of axiom \textbf{A4}, for any regular function $f(x)$
\begin{equation}
\sum_{u=a}^b f(u)=\sum_{u=a}^{-1}f(u)+\sum_{u=0}^b f(u), \; \; \; b\prec a. \label{eq:a_bigger_b}
\end{equation}
\end{remark}

\begin{remark} \label{rem:sum_of_two_intervals}
From axiom \textbf{A4},  it follows that if $G=\mathbb{Z}_{a_1, b_1}\cup \mathbb{Z}_{a_2, b_2}$ and $\mathbb{Z}_{a_1, b_1}\cap \mathbb{Z}_{a_2, b_2}=\emptyset$ then 
$$
\sum_{u\in G}f(u)=\sum_{u=a_1}^{b_1}f(u)+\sum_{u=a_2}^{b_2}f(u).
$$
\end{remark}

We note that our method implies $\sum_{n=\alpha}^{\alpha+k}f(n)=f(\alpha)+f(\alpha+1)+f(\alpha+2)+...+f(\alpha+k), \alpha\in \mathbb{Z}, k\in \mathbb{N}$, 
so we are consistent with the classical definition of summation \cite{brom,knopp}.

The following definition answers the natural question of how the notion of limit is defined.

\begin{definition} \label{def:limit_seq}
A number $A$ is said to be the \textit{limit} of a numeric sequence $F(1)$, $F(2)$,..., $F(n)$,... (function of integer argument), i. e. $\lim_{n\rightarrow\infty}F(n)=A$, if $\sum_{u=1}^{\infty}f(u)=A$, where $F(1)=f(1)$ and $F(u)-F(u-1)=f(u) \; \; (u>1)$, i. e.
$$
F(1)+(F(2)-F(1))+...+(F(n)-F(n-1))+...=A. 
$$
\end{definition}

Note that  definition \ref{def:limit_seq}, in particular, coincides with the classical definition of limit of sequence, which is convergent in usual sense, since given any convergent sequence, one can display its limit as the telescoping series \cite{ash}. This reduces the question of existence of limit of functions of integer argument $F(n)$ to the problem of finding the sum of the series
$$
F(1)+(F(2)-F(1))+...+(F(n)-F(n-1))+...
$$

Because of our postulate, this allows us to claim that any elementary function of integer argument defined on $\mathbb{Z}$ has a definite limit.


\section{Properties of sums over re-ordered number line}\label{sec:sums}

The newly ordered number line brings forth its own properties, while at the same time possesses the usual ones. In this section we consider some of the properties of sums
over re-ordered number line. 

\begin{proposition} \label{prop:a_a-1}
If $f(x)$ is a regular function and $a\in \mathbb{Z}$ is fixed, then 
\begin{equation}
\sum_{u=a}^{a-1}f(u)=\sum_{u\in \mathbb{Z}}f(u).
\end{equation}
\end{proposition}

\begin{proof}
Let us consider two cases. 

(1) $a=0$. \newline 
If $a=0$ then $0\prec -1$. By definition, $\mathbb{Z}_{0,-1}=[0, 1, 2,..., -2, -1]=\mathbb{Z}$. Hence 
$$
\sum_{u=a}^{a-1}f(u)=\sum_{u\in \mathbb{Z}}f(u), \qquad (a=0).
$$

(2) $a\neq 0$. \newline
If $a\neq 0$ then $a-1\prec a$ and, by definition, $\mathbb{Z}_{a, a-1}=[a, -1]\cup [0, a-1]=[0, a-1]\cup [a, -1]=[0, 1, 2,..., -2, -1]$, 
i. e. $\mathbb{Z}_{a, a-1}=\mathbb{Z}$, and we get  
$$
\sum_{u=a}^{a-1}f(u)=\sum_{u\in \mathbb{Z}}f(u), \qquad (a\neq 0).
$$

This concludes the proof. 
\end{proof}

\begin{proposition}\label{prop:sum_zero}
Suppose $f(x)$ is a regular function. Then
\begin{equation}
\sum_{u=a}^{a-1}f(u)=0, \qquad \forall \, a\in \mathbb{Z}  \label{eq:sum_a_a-1_zero}
\end{equation}
or which is the same in view of Proposition \ref{prop:a_a-1}
\begin{equation}
\sum_{u\in \mathbb{Z}}f(u)=0. \label{eq:sum_zero}
\end{equation}
\end{proposition}

\begin{proof}
The proof is done in two steps.

\begin{itemize}
\item[(1)] $a=0$. For $a=0$, we have $a\prec a-1$.
From Definition \ref{def:reg} of regular functions we deduce
\begin{equation}
\sum_{u=a}^{b}f(u)=F(b+1)-F(a), \label{regeq}
\end{equation}
which is true for any values of $a$ and $b$, $a\leq b$.
Taking $a-1$ instead of $b$, we obtain
 
$$
\sum_{u=a}^{a-1}f(u)=F(a)-F(a)=0.
$$

\item[(2)] $a\neq 0$. If $a\neq 0$ then $a-1\prec a$. Taking into account the Remark \ref{rem:a_bigger_b} 
and equality (\ref{regeq}), we obtain 
$$
\sum_{u=a}^{a-1}f(u)=\sum_{u=a}^{-1}f(u)+\sum_{u=0}^{a-1}f(u)=F(0)-F(a)+F(a)-F(0)=0
$$
and
$$
\sum_{u=a}^{a-1}f(u)=0, \quad (a\neq 0).
$$
\end{itemize}
which completes the proof.
\end{proof}

\begin{proposition} \label{prop:pair_of_numbers}
For any numbers $m$ and $n$ such that $m\prec n$
\begin{equation}
\sum_{u=m}^{n}f(u)=\sum_{u=-n}^{-m}f(-u). \label{eq:pair_of_numbers}
\end{equation}
\end{proposition}
\begin{proof}
The proof relies on the fact that the order relation between pairs of numbers $m, n$ and $-n, -m$ is the same.
\end{proof}

\begin{proposition} \label{prop:mean}
Let $f(x)$ be a regular function and let $a$, $b$, $c$ be any integer numbers such that $b\in \mathbb{Z}_{a, c}$. Then 
\begin{equation}
\sum_{u=a}^{c}f(u)=\sum_{u=a}^{b}f(u)+\mathop{{\sum}'}_{u=b+1}^{c}f(u),  \label{eq:prop_mean}
\end{equation}
where the prime on the summation sign means that $\mathop{{\sum}'}_{u=b+1}^{c}f(u)=0$ for $b=c$.
\end{proposition}

\begin{proof}
For $c=b$ or $c=a$ there is nothing to prove. So, let $c\neq b$ and $c\neq a$. The proof consists of two steps.

\begin{itemize}
\item[(1)] $a\prec c$. If $a\prec c$ then $\mathbb{Z}_{a, c}=[a, c]$, that is $[a, c]=[a, b]\cup [b+1, c]$, where $[a, b]\cap [b+1, c]=\emptyset$. 
In view of axiom \textbf{A4}, we obtain
$$
\sum_{u=a}^{c}f(u)=\sum_{u=a}^{b}f(u)+\sum_{u=b+1}^{c}f(u), \quad (b\neq c).
$$

\item[(2)] $c\prec a$. If $c\prec a$ then $\mathbb{Z}_{a, c}=[a, -1]\cup [0, c]$ and, therefore, we have two cases: either $b\in [a, -1]$ or $b\in [0, c]$.

\begin{itemize}
\item[(i)] $b\in [a, -1]$. If $b\in [a, -1]$ then $a\preceq b\preceq -1$, and $[a, -1]=[a, b]\cup [b+1, -1]$, assuming that $[b+1, -1]=\emptyset$ for $b=-1$. Hence we have 
$\mathbb{Z}_{a, c}=[a, -1]\cup [0, c]=[a, b]\cup[b+1, -1]\cup[0, c]=[a, b]\cup([b+1, -1]\cup[0, c])=\mathbb{Z}_{a, b}\cup \mathbb{Z}_{b+1, c}$, 
where $\mathbb{Z}_{a, b}\cap \mathbb{Z}_{b+1, c}=\emptyset$. 
According to the Remark \ref{rem:sum_of_two_intervals}, we get

$$
\sum_{u=a}^{c}f(u)=\sum_{u=a}^{b}f(u)+\sum_{u=b+1}^{c}f(u), \quad (b\neq c).
$$

\item[(ii)] $b\in [0, c]$. If $b\in [0, c]$ then $0\preceq b\prec c$, and $[0, c]=[0, b]\cup [b+1, c]$, that is
$\mathbb{Z}_{a, c}=[a, -1]\cup[0, b]\cup[b+1, c]=([a, -1]\cup[0, b])\cup[b+1, c]=\mathbb{Z}_{a, b}\cup \mathbb{Z}_{b+1, c}$, where $\mathbb{Z}_{a, b}\cap \mathbb{Z}_{b+1, c}=\emptyset$. Hence we obtain
$$
\sum_{u=a}^{c}f(u)=\sum_{u=a}^{b}f(u)+\sum_{u=b+1}^{c}f(u), \quad (b\neq c).
$$
\end{itemize}
\end{itemize}

The proof is completed.
\end{proof}

\begin{corollary}[Proposition \ref{prop:mean}] 
For any regular function $f(x)$
\begin{equation}
\sum_{u\in \mathbb{Z}_{-n,n}}f(u)=\sum_{u=0}^{n}f(u)+\sum_{u=1}^{n}f(-u), \quad \forall \, n.  \label{eq:cor_prop_mean}
\end{equation}
\end{corollary}
\begin{proof}
Substitute $-n$ and $n$ for $a$ and $b$ correspondingly in formula (\ref{def:first_sum_eq}). We have that $\mathbb{Z}_{-n, n}=[-n, -1]\cup [0, n]$ and $[-n, -1]\cap [0, n]=\emptyset$. Then in view of axiom \textbf{A4}, we get 
$$
\sum_{u\in \mathbb{Z}_{-n,n}}f(u)=\sum_{u=0}^{n}f(u)+\sum_{u=-n}^{-1}f(u).
$$
Now using Proposition \ref{prop:pair_of_numbers}, we finally obtain
$$
\sum_{u\in \mathbb{Z}_{-n,n}}f(u)=\sum_{u=0}^{n}f(u)+\sum_{u=1}^{n}f(-u).
$$
\end{proof}

\begin{lemma} \label{lemma3}      
For any regular function $f(x)$
\begin{equation}
\sum_{u=a}^{b}f(u)=-\sum_{u=b+1}^{a-1}f(u), \quad \forall \, a, b\in \mathbb{Z}    \label{eq:lemma3}
\end{equation}
\end{lemma}
\begin{proof}
Letting $c=a-1$ in (\ref{eq:prop_mean}) and taking into account (\ref{eq:sum_a_a-1_zero}), we immediately get (\ref{eq:lemma3}).
\end{proof}

\begin{remark}
It is worth mentioning here that our sums share a  property with integrals: for example, if we exchange upper and lower summation limits, 
as is done above in Lemma \ref{lemma3}, a minus sign is appeared. In order to relate the reversing of summation limits, $\sum_{a+1}^{b}f(\cdot)=-\sum_{b+1}^{a}f(\cdot)$, 
to integrals, one can  adopt the point of view that integration is performed over oriented intervals.  
\end{remark}

The following two lemmas and their corollaries demonstrate some properties of the re-ordered number line and shed some light on its structure.

\begin{lemma} \label{lemma:1} 
We have
\begin{equation}
\lim_{n\rightarrow\infty}(n, -n)=\emptyset.  \label{eq:1_lemma1}
\end{equation}
\end{lemma}

\begin{proof}
We have that $n\prec -n$. Then by definition $\mathbb{Z}_{-n, n}=[-n, -1]\cup [0, n]$. Therefore, for arbitrary number $a$ there exists $n_0=|a|$ such that for every natural number $n\geq n_0$ the relation $a\in \mathbb{Z}_{-n, n}$ holds. Whence it follows that
\begin{equation}
\lim_{n\rightarrow\infty}\mathbb{Z}_{-n, n}=\mathbb{Z}. \label{eq:lemma1}
\end{equation}

But $\mathbb{Z}_{-n, n}=\mathbb{Z}\setminus (n, -n)$. Therefore, passing to the limit and taking into account (\ref{eq:lemma1}), we obtain 
$$
\lim_{n\rightarrow\infty}(n, -n)=\emptyset.
$$
This completes the proof.
\end{proof}

Relying on (\ref{eq:lemma1}) and using the analogy of correspondence of partial sums to infinite series, we obtain
\begin{equation}
\lim_{n\rightarrow\infty}\sum_{u\in \mathbb{Z}_{-n, n}}f(u)=\sum_{u\in \mathbb{Z}}f(u). \label{eq:partial_sum}
\end{equation}

\begin{lemma} \label{lemma2} 
For any elementary function $F(x)$ defined on $\mathbb{Z}$
\begin{equation}
\lim_{n\rightarrow\infty}F(n+1+\alpha)=\lim_{n\rightarrow\infty}F(-n+\alpha), \quad \forall \, \alpha\in\mathbb{R}   \label{eq:lemma2}
\end{equation}
\end{lemma}

\begin{proof}
Let $F(x+1+\alpha)-F(x+\alpha)=f(x)$. According to (\ref{eq:a_bigger_b}), for any natural number $n$
$$
\sum_{u\in \mathbb{Z}_{-n, n}}f(u)=\sum_{u=-n}^{-1}f(u)+\sum_{u=0}^{n}f(u).
$$
Hence, in view of  (\ref{regeq})
$$
\sum_{u\in \mathbb{Z}_{-n, n}}f(u)=F(\alpha)-F(-n+\alpha)+F(n+1+\alpha)-F(\alpha)=F(n+1+\alpha)-F(-n+\alpha).
$$
By definition, the function $f(x)$ is regular. Thus, passing to the limit, $n\rightarrow\infty$, and taking into account (\ref{eq:sum_zero}) and (\ref{eq:partial_sum}), 
we obtain
$$
\lim_{n\rightarrow\infty}F(n+1+\alpha)=\lim_{n\rightarrow\infty}F(-n+\alpha),
$$
which completes the proof.
\end{proof}

\begin{corollary}[Lemma \ref{lemma2}]
Let $F(x)$ be an elementary function defined on $\mathbb{Z}$. Then
$$
\lim_{n\rightarrow\infty}F(-(n+1))=\lim_{n\rightarrow\infty}F(n).
$$
\end{corollary}

In particular,
\begin{equation}
\lim_{n\rightarrow\infty}(-(n+1))=\lim_{n\rightarrow\infty}n  \label{eq:1_cor_lemma2}
\end{equation}

\begin{equation}
\lim_{n\rightarrow\infty}(n+1)=\lim_{n\rightarrow\infty}(-n),   \label{eq:2_cor_lemma2}
\end{equation}
from which we find
$$
\lim_{n\rightarrow\infty}n=-\frac{1}{2}.
$$

Since $n=\sum_{u=1}^{n}1 \quad \forall n$, then, using axiom \textbf{A1}, $\lim_{n\rightarrow\infty}n=\sum_{u=1}^{\infty}1$
and
\begin{equation}
\sum_{u=1}^{\infty}1=-\frac{1}{2}.  \label{eq:sum_of_1}
\end{equation}


\section{Main theorems}\label{sec:main}

We present a number of general theorems on series summation, several theorems concerning operations with series, and some propositions on limits of unbounded and oscillating functions. The theorems are useful in applications.


\begin{theorem} \label{theorem:inf_sum}
Let $f(x)$ be a regular function such that $f(-x)=f(x)$. Then
\begin{equation}
\sum_{u=1}^{\infty}f(u)=-\frac{f(0)}{2}  \label{eq:inf_sum}
\end{equation}
independently of whether the series is convergent or not in usual sense.
\end{theorem}

\begin{proof}
We present two alternative proofs.
\begin{enumerate}
\item Putting $a=0$ and $b=-n$ in (\ref{eq:lemma3}), we have

$$
\sum_{u=0}^{-n}f(u)=-\sum_{u=-n+1}^{-1}f(u).
$$

Then using  
Proposition \ref{prop:pair_of_numbers}, we get

$$
\sum_{u=0}^{-n}f(u)=-\sum_{u=1}^{n-1}f(-u)
$$
or
\begin{equation}
\sum_{u=1}^{-n}f(u)=-\sum_{u=1}^{n-1}f(-u)-f(0). \label{eq:1_from_lemma2}
\end{equation}

Now substitute $n$ by $n+1$ in (\ref{eq:1_from_lemma2}). Then passing to the limit with regard to (\ref{eq:1_cor_lemma2}), we obtain 

\begin{equation}
\sum_{u=1}^{\infty}f(u)=-\sum_{u=1}^{\infty}f(-u)-f(0).   \label{eq:2_from_lemma3}
\end{equation}

From (\ref{eq:2_from_lemma3}),  taking into account  $f(-x)=f(x)$, the formula (\ref{eq:inf_sum}) follows.

\item According to (\ref{eq:a_bigger_b}), for any natural $n$ we have
$$
\sum_{u=-n}^{n}f(u) = \sum_{u=-n}^{-1}f(u) + \sum_{u=0}^{n}f(u) = \sum_{u=1}^{n}f(u) + f(0) + \sum_{u=-n}^{-1}f(u).
$$
Replace the last sum using Proposition \ref{prop:pair_of_numbers} and get
$$
\sum_{u=-n}^{n}f(u) = \sum_{u=1}^{n}f(u) + f(0) + \sum_{u=1}^{n}f(-u)
$$
By Definintion \ref{def:sum} and since $f(-x)=f(x)$, we have
$$
\sum_{u\in\mathbb{Z}_{-n,n}}f(u) = 2 \sum_{u=1}^{n}f(u) + f(0)
$$
Now passing to the limit and taking into account equations (\ref{eq:sum_zero}) and (\ref{eq:partial_sum}), we obtain the formula (\ref{eq:inf_sum}).
\end{enumerate}
The theorem is proved. 
\end{proof}

Below we present some examples to the above theorem.

\begin{example}

\begin{itemize}
\item[(1)] Convergent series 
\begin{align*}
&\sum_{u=1}^{\infty}\frac{1}{4u^2-1}=\frac{1}{2}, && F(n)=\frac{-1}{2(2n-1)} && \textup{\cite{grad}} \\
&\sum_{u=1}^{\infty}(-1)^{u}\frac{2u^2+1/2}{(2u^2-1/2)^2}=-1, && F(n)=\frac{(-1)^{n-1}}{(2n-1)^2} \\
&\sum_{u=1}^{\infty}\frac{(4^u-1)(u-1/2)-1}{2^{u^2+u+1}}=\frac{1}{4}, && F(n)=\frac{-(n-1/2)}{2^{n^2-n+1}}
\end{align*}
\begin{multline*}
\sum_{u=1}^{\infty}\frac{(u^2+1/4)\tan(1/2)\cos u-u\sin u}{(4u^2-1)^2}=-\frac{\tan(1/2)}{8}, \\
F(n)=\frac{\sin(n-1/2)}{8(2n-1)^2\cos(1/2)} 
\end{multline*}

\item[(2)] Divergent series
\begin{align}
&1-1+1-1+...=\frac{1}{2}, && F(n)=\frac{(-1)^n}{2} &&  \textup{\cite{hardy}} \notag \\
&1+1+1+1+...=-\frac{1}{2}, && F(n)=n-1  && \textup{\cite{titch}} \label{eq:sum_of_ones} \\
&1^{2k}+2^{2k}+3^{2k}...=0, \quad \forall \, k\in\mathbb{N}, && F(n)=B_{2k}(n-1)  && \textup{\cite{titch}} \label{eq:sum_even_powers}
\end{align}
\begin{align}
&1^{2k}-2^{2k}+3^{2k}-...=0, \quad \forall \, k\in\mathbb{N},  && \textup{\cite{hardy}} \label{eq:sum_alt_even_powers}\\
&F(n)=\frac{(-1)^n}{2k+1}\sum_{u=1}^{2k+1}(2^u-1)\binom{2k+1}{u}B_{u}(n-1)^{2k+1-u} \notag
\end{align}
\begin{align}
&\sum_{u=1}^{\infty}(-1)^{u-1}u^{2k-1}\sin u\theta=0, &&  \forall \, k\in\mathbb{N}, && -\pi<\theta<\pi && \textup{\cite{hardy}}  \label{ex:diver_sin} \\
&\sum_{u=1}^{\infty}(-1)^{u-1}u^{2k}\cos u\theta=0, && \forall \, k\in\mathbb{N}, && -\pi<\theta<\pi && \textup{\cite{hardy}} \notag
\end{align}
with the primitive functions for the last two formulas, respectively
\begin{multline*}
F(n)=  
(-1)^n\Biggl[\sum_{u=1}^{k}\beta_{u}\left(n-\frac{1}{2}\right)^{2u-1}\sin\left(n-\frac{1}{2}\right)\theta + \\
 \sum_{u=0}^{k-1}\bar{\beta_{u}}\left(n-\frac{1}{2}\right)^{2u}\cos\left(n-\frac{1}{2}\right)\theta\Biggr]
\end{multline*}
\begin{multline*}
F(n)=   
 (-1)^n\Biggl[\sum_{u=1}^{k}\gamma_{u}\left(n-\frac{1}{2}\right)^{2u-1}\sin\left(n-\frac{1}{2}\right)\theta + \\
\sum_{u=0}^{k}\bar{\gamma_{u}}\left(n-\frac{1}{2}\right)^{2u}\cos\left(n-\frac{1}{2}\right)\theta\Biggr]  
\end{multline*}
where the coefficients $\beta_u, \bar{\beta_u}, \gamma_u, \bar{\gamma_u}$ are taken in such a way so that
$$
F(n+1)-F(n)=(-1)^{n-1}n^{2k-1}\sin n\theta
$$
and
$$
F(n+1)-F(n)=(-1)^{n-1}n^{2k}\cos n\theta
$$
\end{itemize}

In particular, for $\theta=\frac{\pi}{2}$ in (\ref{ex:diver_sin}), we get
\begin{align}
&1^{2k-1}-3^{2k-1}+5^{2k-1}-...=0, \quad \forall \, k \in\mathbb{N} && \textup{\cite{hardy}}  \label{eq:from_diver_sin}
\end{align}
\end{example}

\begin{corollary}[Theorem \ref{theorem:inf_sum}]
For any regular function $f(x)$ 
\begin{equation}
\sum_{u=1}^{\infty}\left(f(u)+f(-u)\right)=-f(0)  \label{eq:1_cor_theor_inf_sum}
\end{equation}
independently of whether the series is convergent or not.
\end{corollary}

\begin{example}
It is easy to check that the function $f(x)=1/(9x^2-3x-2)$ is regular and 
$$
f(x)+f(-x)=\frac{18x^2-4}{81x^4-45x^2+4}.
$$
Hence, the sum of convergent series
$$
\sum_{u=1}^{\infty}\frac{18u^2-4}{81u^4-45u^2+4}=\frac{1}{2}.
$$
In particular, if we take $f(x)=e^x$, we get
$$
\sum_{u=1}^{\infty}(e^u+e^{-u})=-1
$$
and since $\cosh x=(e^x+e^{-x})/2$ we  obtain
$$
\sum_{u=1}^{\infty}\cosh u=-\frac{1}{2}.
$$
\end{example}

The theorem \ref{theorem:inf_sum} can be generalized as follows. We first introduce the notion of \textit{quasi-even} function.

\begin{definition} \label{def:quasi-even}
The function $f(x)$, $x\in\mathbb{R}$, is called \textsl{quasi-even} if it satisfies the condition $f(-x)=f(x-a)$, $a\in\mathbb{Z}$.
\end{definition}

\begin{theorem}\label{theor:gen_inf_sum}
Let $f(x)$ be a regular quasi-even function that satisfies the condition of definition \ref{def:quasi-even} with $a=\epsilon t$, where $\epsilon=\pm 1$
and $t$ is a fixed natural number, and let $\delta=2^{-1}(1-\epsilon)$. Then we have
\begin{equation}
\sum_{u=1}^{\infty}f(u)=\frac{1}{2}\,\epsilon\sum_{u=\delta}^{t-1+\delta}\left(\lim_{n\rightarrow\infty}f(n-\epsilon u)-f(-\epsilon u)\right)-\frac{1}{2}f(0). \label{eq:gen_inf_sum}
\end{equation}
\end{theorem}
\begin{proof}
According to formula (\ref{eq:1_from_lemma2}),
$$
\sum_{u=1}^{-(n+1)}f(u)=-\sum_{u=0}^{n}f(-u).
$$
In view of $f(-x)=f(x-\epsilon t)$, we have
$$
\sum_{u=0}^{n}f(-u)=\sum_{u=0}^{n}f(u-\epsilon t)=\sum_{u=1}^{n}f(u)-\epsilon\sum_{u=\delta}^{t-1+\delta}f(n-\epsilon u)+\epsilon\sum_{u=\delta}^{t-\delta}f(-\epsilon u).
$$
Since 
$$
\sum_{u=\delta}^{t-\delta}f(-\epsilon u)=\sum_{u=\delta}^{t-1+\delta}f(-\epsilon u)+\epsilon f(-\epsilon t)
$$ 
and $f(-\epsilon t)=f(0)$, we get
$$
\sum_{u=0}^{n}f(-u)=\sum_{u=1}^{n}f(u)-\epsilon\sum_{u=\delta}^{t-1+\delta}\left(f(n-\epsilon u)-f(-\epsilon u)\right)+f(0).
$$
Hence
$$
\sum_{u=1}^{-(n+1)}f(-u)=-\sum_{u=1}^{n}f(u)+\epsilon\sum_{u=\delta}^{t-1+\delta}\Bigl(f(n-\epsilon u)-f(-\epsilon u)\Bigr)-f(0).
$$
Passing to the limit and taking into account (\ref{eq:1_cor_lemma2}), we finally obtain
$$
\sum_{u=1}^{\infty}f(u)=\frac{\epsilon}{2}\sum_{u=\delta}^{t-1+\delta}\Bigl(\lim_{n\rightarrow\infty}f(n-\epsilon u)-f(-\epsilon u)\Bigr)-\frac{1}{2}f(0).
$$
and the proof is completed.
\end{proof}

For the special case of $t=1$ the formula (\ref{eq:gen_inf_sum}) reduces to
$$
\sum_{u=1}^{\infty}f(u)=\frac{1}{2}\lim_{n\rightarrow\infty}f(n)-f(0) \qquad (\epsilon=1),
$$
\begin{equation}
\sum_{u=1}^{\infty}f(u)=-\frac{1}{2}\lim_{n\rightarrow\infty}f(n+1) \qquad (\epsilon=-1)  \label{eq:1_theor_gen_inf_sum}
\end{equation}

Formally, for $\epsilon=0$, the formula (\ref{eq:gen_inf_sum}) coincides with the formula (\ref{eq:inf_sum}). So, the theorem \ref{theor:gen_inf_sum} generalizes the theorem \ref{theorem:inf_sum} for the class of quasi-even functions.

Using the proof of theorem \ref{theor:gen_inf_sum}, we obtain the following

\begin{corollary}[Theorem \ref{theor:gen_inf_sum}] \label{cor:gen_theor}
For any regular function $f(x)$ we have 
\begin{equation}
\sum_{u=1}^{\infty}f(u-\epsilon t)=\sum_{u=1}^{\infty}f(u)-\epsilon \sum_{u=\delta}^{t-1+\delta}\left(\lim_{n\rightarrow\infty}f(n-\epsilon u)-f(-\epsilon u)\right). \label{eq:cor_gen_theor}
\end{equation}
\end{corollary}

From the corollary \ref{cor:gen_theor}, if we suppose that $\varphi_1(x), \varphi_2(x),\dots, \varphi_r(x)$ are regular functions, $\alpha_1$, $\alpha_2$,..., $\alpha_r$ are real numbers 
and $\sum_{i=1}^{r}\alpha_i \: \varphi_i(x)=f(x-\epsilon t)$, we get the formula
\begin{equation} \label{eq:cor_sum}
\begin{split}
\sum_{u=1}^{\infty}f(u)=\alpha_1 \sum_{u=1}^{\infty}\varphi_1(u)+\alpha_2 \sum_{u=1}^{\infty}\varphi_2(u)+...
+ \alpha_r \sum_{u=1}^{\infty}\varphi_r(u)+  \\
+\;
\epsilon \sum_{u=\delta}^{t-1+\delta}\left(\lim_{n\rightarrow\infty}f(n-\epsilon u)-f(-\epsilon u)\right). 
\end{split}
\end{equation}


Below we establish several new properties of divergent series, as analogous to those which are known for convergent series.

\begin{theorem} \label{theor:commut}
For any regular function $f(x)$
\begin{equation}
\sum_{u=1}^{\infty}f(u)=\sum_{u=1}^{\infty}f(2u-1)+\sum_{u=1}^{\infty}f(2u)+\frac{1}{2}\lim_{n\rightarrow\infty}\left(f(n)-(-1)^nf(n)\right). \label{eq:theor_commut}
\end{equation}
\end{theorem}
\begin{proof}
Let us take the equality
$$
\sum_{u=1}^{n}f(u)=\sum_{u=1}^{[n/2]}f(2u-1)+\sum_{u=1}^{[n/2]}f(2u)+\frac{1}{2}\left(f(n)-(-1)^nf(n)\right), \qquad \forall \, n
$$
Now, passing to the limit and taking into account axioms \textbf{A1} and \textbf{A2} we arrive at (\ref{eq:theor_commut}). This completes the proof.
\end{proof}

\begin{corollary}[Theorem \ref{theor:commut}] \label{cor:theor_commut}
If \: $\lim_{n\rightarrow\infty}(-1)^{n}f(n)=0$, then
$$
\sum_{u=1}^{\infty}f(u)=\sum_{u=1}^{\infty}f(2u-1)+\sum_{u=1}^{\infty}f(2u)+\frac{1}{2}\lim_{n\rightarrow\infty}f(n).
$$
\end{corollary}

\begin{theorem} \label{theor:assoc}
For any regular function $f(x)$
\begin{equation}
\sum_{u=1}^{\infty}f(u)=\sum_{u=1}^{\infty}\left(f(2u-1)+f(2u)\right)+\frac{1}{2}\lim_{n\rightarrow\infty}\left(f(n)+(-1)^{n}f(n)\right). \label{eq:theor_assoc}
\end{equation}
\end{theorem}
\begin{proof}
The proof is analogous to that of theorem \ref{theor:commut}.
\end{proof}

\begin{example}
Consider a couple of simple examples illustrating the above theorem.
\begin{enumerate}
\item Let $f_1(n)=1$ and $f_2(n)=(-1)^{n-1}$. Then according to theorem \ref{theor:assoc}, we have
$$ 
\sum_{u=1}^{\infty}f_1(u)=\sum_{u=1}^{\infty}1=\sum_{u=1}^{\infty}(1+1)+\frac{1}{2}\lim_{n\to\infty}\left(1+(-1)^{n}\right)
$$
and
$$
\sum_{u=1}^{\infty}f_2(u)=\sum_{u=1}^{\infty}(-1)^{u-1}=\sum_{u=1}^{\infty}(1-1)+\frac{1}{2}\lim_{n\to\infty}\left((-1)^{n}+1\right).
$$
But $\lim_{n\to\infty}(-1)^{n}=0$ due to (\ref{propmu}), hence we find
$$
\sum_{u=1}^{\infty}1=2\sum_{u=1}^{\infty}1+\frac{1}{2}, ~~~\mathrm{or}~~ \sum_{u=1}^{\infty}1=-\frac{1}{2}
$$
and 
$$
\sum_{u=1}^{\infty}(-1)^{u-1}=0+\frac{1}{2}=\frac{1}{2}.
$$
These two identities are in agreement with the formula (\ref{eq:inf_sum}) of theorem \ref{theorem:inf_sum}.
\item Let $f(n)=(-1)^{n-1}\,n$. Then agan according to theorem \ref{theor:assoc}, we have
\begin{multline*} 
\sum_{u=1}^{\infty}f(u)=\sum_{u=1}^{\infty}(-1)^{n-1}\,n = 
\sum_{u=1}^{\infty}\left((-1)^{2u-2}\,(2u-1)+(-1)^{2u-1}\,2u\right)  \\
+\frac{1}{2}\lim_{n\to\infty}\left((-1)^{n-1}\,n+n\right) 
\end{multline*}
or
$$ 
\sum_{u=1}^{\infty}(-1)^{n-1}\,n = 
\sum_{u=1}^{\infty}(-1)+\frac{1}{2}\lim_{n\to\infty}\left((-1)^{n-1}\,n+n\right).
$$
But from theorem \ref{theor:polynom1} that we establish below we have $\lim_{n\to\infty}(-1)^{n-1}\,n=0$ and
$$ 
\sum_{u=1}^{\infty}(-1)^{n-1}\,n = 
-\sum_{u=1}^{\infty}1+\frac{1}{2}\lim_{n\to\infty}n.
$$
From (\ref{eq:2_cor_lemma2}) we have $\lim_{n\to\infty}n=-1/2$, and finally we obtain
$$
\sum_{u=1}^{\infty}(-1)^{n-1}\,n = 1-2+3-4+\dots = \frac{1}{4}.
$$
The same result follows immediately from the theorem \ref{theor:alt_arith} below.
\end{enumerate}
\end{example}

\begin{remark}
We would like to note that our summation method allows us to write 
$\left(1-1+1-1+\dots\right)^2=1-2+3-4+\dots$, since both sides are equal to 1/4 with our method.
So, the last formula in the above example can be obtained if we raise both sides of $1-1+1-1+\dots=1/2$ to the second power, 
or, in other words, if we consider the Cauchy product of $(-1)^n$ with itself.
Indeed, the Cauchy product of two infinite series 
$$
\left(\sum_{n=0}^{\infty}f(n)\right)\cdot\left(\sum_{n=0}^{\infty}g(n)\right)=\sum_{n=0}^{\infty}h(n),
~~\mathrm{where}~~ h(n)=\sum_{k=0}^{n}f(k)g(n-k), 
$$
for $n=0,1,2,\dots$, is defined even when both of the series are divergent. 
Taking $f_1(n)=f_2(n)=(-1)^n$, we have
$$
h(n)=\sum_{k=0}^{n}f_1(k)f_2(n-k)=\sum_{k=0}^{n}(-1)^k(-1)^{n-k}=(-1)^n(n+1)
$$
such that the product series
$$
\sum_{n=0}^{\infty}h(n)=\sum_{n=0}^{\infty}(-1)^n(n+1)=1-2+3-4+\dots
$$
Thus we get that
$\left(1-1+1-1+\dots\right)^2=1-2+3-4+\dots$
with both the sides equal to 1/4.
As is known, summation methods that are linear, stable, and respect the Cauchy product of two series and sums $1-1+1-1+\dots=1/2$, should also sum $1-2+3-4+\dots = 1/4$, as our method does.
\end{remark}

\begin{theorem} \label{theor:distrib}
Let $f(x)$ be a regular function and let $\sum_{u=1}^{\infty}f(u)=A$ and $\sum_{u=1}^{\infty}(-1)^{u-1}f(u)=B$. Then
\begin{equation}
\sum_{u=1}^{\infty}f(2u)=\frac{1}{2}(A-B)  \label{eq:1_theor_distrib}
\end{equation}
and
\begin{equation}
\sum_{u=1}^{\infty}f(2u-1)=\frac{1}{2}(A+B)-\frac{1}{2}\lim_{n\rightarrow\infty}\left(f(n)-(-1)^nf(n)\right). \label{eq:2_theor_distrib}
\end{equation}
\end{theorem}
\begin{proof}
We have the following equalities
$$
\sum_{u=1}^{[n/2]}f(2u)=\frac{1}{2}\left[\sum_{u=1}^{n}f(u)-\sum_{u=1}^{n}(-1)^{u-1}f(u)\right], \quad \forall \, n
$$
and
$$
\sum_{u=1}^{[n/2]}f(2u-1)=\frac{1}{2}\left[\sum_{u=1}^{n}f(u)+\sum_{u=1}^{n}(-1)^{u-1}f(u)-\left(f(n)-(-1)^nf(n)\right)\right], \quad \forall \, n.
$$
Then taking the limit, and relying on axioms \textbf{A1} and \textbf{A2}, we obtain the formulas (\ref{eq:1_theor_distrib}) and (\ref{eq:2_theor_distrib}). 
\end{proof}

Using the proofs of theorems \ref{theor:commut} and \ref{theor:distrib} we get the following
\begin{proposition}
For any regular function $f(x)$ such that $f(2x)=d \, f(x), \; d\in\mathbb{R}$, we have
$$
\sum_{u=1}^{\infty}(-1)^{u-1}f(u)=(1-2d) \sum_{u=1}^{\infty}f(u)
$$
and
$$
\sum_{u=1}^{\infty}f(2u-1)=(1-d) \sum_{u=1}^{\infty}f(u) - \frac{1}{2}\lim_{n\rightarrow\infty}\left(f(n)-(-1)^nf(n)\right).
$$
\end{proposition}

\begin{example}
The function $f(n)=n^k$ is regular, as we have $B_{k}(n)-B_{k}(n-1)=n^k$, where 
$$
B_{k}(n)=\frac{1}{k+1}\,\sum_{u=0}^{k}\binom{k+1}{u}B_{u}\,n^{k+1-u} = \sum_{u=1}^{n}u^{k}
$$ 
is the Bernoulli polynomial.

For the values of Riemann zeta function and Dirichlet eta function at negative integers, we know that
$$
\zeta(-k)=\sum_{u=1}^{\infty}u^k=-\frac{B_{k+1}}{k+1} 
~~\mathrm{and}~~ 
\eta(-k)=\sum_{u=1}^{\infty}(-1)^{u-1}u^k=\frac{2^{k+1}-1}{k+1}B_{k+1},
$$
where $B_k$ is the Bernoulli number.

Under the summation sign we have a regular function $n^k$ and we apply theorem \ref{theor:distrib} to find the sums
for the series of $f(2n)$ and $f(2n-1)$. We have
$$
\sum_{u=1}^{\infty}(2u)^k = 2^k+4^k+\dots = \frac{1}{2}\left(-\frac{B_{k+1}}{k+1}-\frac{B_{k+1}}{k+1}\left(2^{k+1}-1\right)\right)
 = -\frac{2^k}{k+1}B_{k+1}.
$$
Particularly, for even $k=2m, \, m\in\mathbb{N}$, we get 
$$
\sum_{u=1}^{\infty}(2u)^{2m} = 2^{2m}+4^{2m}+\dots = -\frac{2^{2m}}{2m+1}B_{2m+1}=0,
$$
which confirms the formula (\ref{eq:sum_even_powers}). For odd $k=2m-1$, we get the formula
$$
\sum_{u=1}^{\infty}(2u)^{2m-1} = 2^{2m-1}+4^{2m-1}+\dots = -\frac{2^{2m-1}}{2m}B_{2m}.
$$

For the series of $f(2n-1)$ we have
\begin{multline*}
\sum_{u=1}^{\infty}(2u-1)^k = 1^k+3^k+5^k+\dots = \frac{1}{2}\left(-\frac{B_{k+1}}{k+1}+\frac{B_{k+1}}{k+1}\left(2^{k+1}-1\right)\right) \\
 - \frac{1}{2}\lim_{n\to\infty}\left(n^k-(-1)^n n^k\right). 
\end{multline*}
To find the limits, we use the statements of theorem \ref{theor:polynom1} and theorem \ref{theor:polynom2} and get
\begin{multline*}
\sum_{u=1}^{\infty}(2u-1)^k = 1^k+3^k+5^k+\dots = \frac{B_{k+1}}{k+1}\left(2^k-1\right) - \frac{1}{2}\frac{(-1)^k}{k+1}  \\
 = \frac{2(2^k-1)B_{k+1}-(-1)^k}{2(k+1)}. 
 \end{multline*}
In particular, for even $k=2m, \, m\in\mathbb{N}$, we get
$$
\sum_{u=1}^{\infty}(2u-1)^{2m} = 1^{2m}+3^{2m}+5^{2m}+\dots = -\frac{1}{2(2m+1)}.
$$
For odd $k=2m-1$, we obtain the formula
$$
\sum_{u=1}^{\infty}(2u-1)^{2m-1} = 1^{2m-1}+3^{2m-1}+5^{2m-1}+\dots = \frac{2(2^{2m-1}-1)B_{2m}+1}{4m}.
$$
Actually, the above formulas are also true for $k=0$. The special case of $k=1$, as well as $k=0$, is in agreement
with the statement of theorem \ref{theor:arith}. 
Applying similar techniques, one can find many other sum formulas for series of this type.
\end{example}

\begin{remark}
From the above example, knowing that $f(x)=x^{k}$ is regular and using (\ref{eq:1_from_lemma2}), we can deduce that all odd Bernoulli numbers, except for $B_1$, equals to zero. 
Indeed, from (\ref{eq:1_from_lemma2}) we have $B_{k}(-n)=(-1)^{k-1}B_{k}(n-1)$, where $B_{k}(n)$ is the Bernoulli polynomial. Therefore,
$$
B_{k}(n)-(-1)^{k-1}B_{k}(-n)=\frac{2}{k+1}\,\sum_{u=0}^{\left[2^{-1}(k-1)\right]}\binom{k+1}{2u+1}B_{2u+1}\,n^{k-2u}=n^{k},
$$
whence it immediately follows that $B_{1}=1/2$ and $B_{2u+1}=0, \ u=1,2,...$.
\end{remark}


Further in the text we will need some statements related to the limits of functions. In section \ref{sec2} we claimed that every function defined on the new number line has a definite limit. Below we give some propositions concerning the limiting behavior of unbounded and oscillating functions.

\begin{theorem} \label{theorem:theta1}
Let $\theta(x)$ be an elementary function such that 
\begin{equation}
\theta(-x)=-\theta(x-\epsilon), \label{eq:1_theor_theta1}
\end{equation} 
where $\epsilon=\{0,1,-1\}$, and let $\delta=(1-\epsilon)/2$.

Then
\begin{equation}
\lim_{n\rightarrow\infty}\theta(n+\delta)=0.  \label{eq:2_theor_theta1}
\end{equation}
\end{theorem}

\begin{proof}
We put $\theta(x+\delta)-\theta(x-1+\delta)=f(x)$. Then, according to the formula (\ref{regeq}), for any natural number $n$
$$
\theta(n+\delta)-\theta(\delta)=\sum_{u=1}^{n}f(u).
$$
The function $f(x)$ is regular and even, $f(-x)=f(x)$. 
Therefore, taking the limit and using (\ref{eq:inf_sum}), we get
$$
\lim_{n\rightarrow\infty}\theta(n+\delta)-\theta(\delta)=\lim_{n\rightarrow\infty}\sum_{u=1}^{n}f(u)=\sum_{u=1}^{\infty}f(u)=-\frac{f(0)}{2}
$$
and since $f(0)=\theta(\delta)-\theta(-1+\delta)=\theta(\delta)+\theta(\delta)=2\theta(\delta)$, we finally obtain
$$
\lim_{n\rightarrow\infty}\theta(n+\delta)=0,
$$
which completes the proof.
\end{proof}


\begin{proposition} \label{prop:odd_psi}
For any odd elementary function $\psi(x)$ we have
\begin{equation}
\lim_{n\rightarrow\infty}\sum_{u=\delta}^{t-1+\delta}\psi(n+\frac{\epsilon t}{2}-\epsilon u)=0  \label{eq:prop_odd_psi}
\end{equation}
where \ $\epsilon = \{1, -1\}$, $\delta=(1-\epsilon)/2$ and $t\in\mathbb{N}$ is fixed.
\end{proposition}

\begin{proof}
Let us take an elementary function $\omega(x)$ such that
\begin{equation}
\omega(-x)=-\omega(x-\epsilon t).  \label{eq:rem_omega1}
\end{equation}
Then, the expression
$$
H_{\omega}(x)=\sum_{u=\delta}^{t-1+\delta}\omega(x-\epsilon u)
$$
satisfies the condition
$H_{\omega}(-x)=-H_{\omega}(x-1)$
and according to theorem \ref{theorem:theta1} 
$$
\lim_{n\rightarrow\infty}H_{\omega}(n)=0.
$$
The function $\psi\left(x+\epsilon t/2\right)$, where $\psi(x)$ is an odd elementary function, satisfies the condition (\ref{eq:rem_omega1}). Indeed,
$$
\psi\left(-x+\frac{\epsilon t}{2}\right)=\psi\left(-\left(x-\frac{\epsilon t}{2}\right)\right)=-\psi\left(x-\frac{\epsilon t}{2}\right).
$$
Therefore, taking $\psi\left(x+\epsilon t/2\right)$ instead of $\omega(x)$ we obtain (\ref{eq:prop_odd_psi}).
\end{proof}

\begin{theorem} \label{theor:vartheta}
Let $\vartheta(x)$ be an elementary function such that
\begin{equation}
\vartheta(-x)=\vartheta(x-\epsilon t)  \label{eq:theor_vartheta}
\end{equation}
where \ $\epsilon = \{1, -1\}$, $\delta=(1-\epsilon)/2$ and $t\in\mathbb{N}$ is fixed.

Then we have
$$
\lim_{n\rightarrow\infty}(-1)^{n}\vartheta(n+\delta)=0
$$
and
$$
\lim_{n\rightarrow\infty}(-1)^{n}\sum_{u=\delta}^{t-1+\delta}\vartheta(n-\epsilon u)=0.
$$
\end{theorem}
\begin{proof}
The proof is in the same way and uses the same techniques as for the previous two propositions.
\end{proof}

\begin{corollary}[Theorem \ref{theor:vartheta}] \label{cor:theor_vartheta}
For any even elementary function $\mu(x)$
\begin{equation}
\lim_{n\rightarrow\infty}(-1)^{n}\sum_{u=\delta}^{t-1+\delta}\mu(n+\frac{\epsilon t}{2}-\epsilon u)=0, \qquad (\epsilon=\pm 1). \label{eq:cor_theor_vartheta}
\end{equation}
\end{corollary}

\begin{example}
Let us take the equality
$$
\sum_{u=1}^{n}\cos u\theta=-\frac{1}{2}+\left(2\sin\frac{\theta}{2}\right)^{-1}\sin\left(n+\frac{1}{2}\right)\theta
$$
which holds for any natural $n$ and real $\theta$. Then, according to axiom \textbf{A1}, we get
\begin{equation}
\sum_{u=1}^{\infty}\cos u\theta=-\frac{1}{2}+\left(2\sin\frac{\theta}{2}\right)^{-1}\lim_{n\rightarrow\infty}\sin\left(n+\frac{1}{2}\right)\theta. 
\label{cos}
\end{equation}
From the formula (\ref{eq:prop_odd_psi}) for $t=1$ and $\epsilon=1$, we find 
$$
\lim_{n\rightarrow\infty}\psi\left(n+\frac{1}{2}\right)=0.
$$
Now taking this into account in (\ref{cos}), we obtain
\begin{equation}
\sum_{u=1}^{\infty}\cos u\theta=-\frac{1}{2}. \qquad\qquad\qquad\qquad\qquad \textup{\cite{hardy}} \label{ex:cos}
\end{equation}
Then taking the trigonometric identity
$$
\sum_{u=1}^{n}(-1)^{u-1}\cos u\theta=\frac{1}{2}-\left(2\cos\frac{\theta}{2}\right)^{-1}(-1)^{n}\cos\left(n+\frac{1}{2}\right)\theta
$$
and passing to the limit, we get
\begin{equation}
\sum_{u=1}^{\infty}(-1)^{u-1}\cos u\theta=\frac{1}{2}-\left(2\cos\frac{\theta}{2}\right)^{-1}\lim_{n\rightarrow\infty}(-1)^{n}\cos\left(n+\frac{1}{2}\right)\theta. \quad \textup{\cite{hardy}}  \label{alt_cos}
\end{equation}
From the formula (\ref{eq:cor_theor_vartheta}) for $t=1$ and $\epsilon=1$, we find 
$$
\lim_{n\rightarrow\infty}(-1)^{n}\,\mu\left(n+\frac{1}{2}\right)=0.
$$
Now taking this into account in (\ref{alt_cos}), we get
\begin{equation}
\sum_{u=1}^{\infty}(-1)^{u-1}\cos u\theta=\frac{1}{2} \qquad\qquad\qquad\qquad\qquad \textup{\cite{hardy}}  \label{ex:alt_cos}
\end{equation}
\end{example}
The formulas (\ref{ex:cos}) and (\ref{ex:alt_cos}) can also be obtained with use of (\ref{eq:inf_sum}).


\begin{theorem} \label{theor:polynom1}
Suppose $f(x)$ is a polynomial defined over the field of real numbers, $x\in \mathbb{R}$. Then we have
\begin{equation}
\lim_{n\rightarrow\infty}(-1)^{n}f(n)=0.  \label{eq:theor_polynom1}
\end{equation}
\end{theorem}

To prove the Theorem \ref{theor:polynom1}, we need the following

\begin{proposition} \label{prop:lim}
For any natural number $k$
\begin{equation}
\lim_{n\rightarrow\infty}(-1)^{n}(2n+1)^{k}=0. \label{limzero}
\end{equation}
\end{proposition}

The complete proofs of theorem \ref{theor:polynom1} and proposition \ref{prop:lim} can be found in \cite{bag-jmr}. 
Nevertheless, we provide the proofs here for the sake of completeness.

\begin{proof}[Proof of Proposition \ref{prop:lim}]
Let us consider the cases of $k$ even and $k$ odd.
\begin{enumerate}
	\item Let $k$ is even, $k=2m$. \newline  
	From the formula (\ref{eq:cor_theor_vartheta}) for $\epsilon=1$ and $t=1$, we have
	\begin{equation}
	\lim_{n\rightarrow\infty}(-1)^{n}\mu\left(n+\frac{1}{2}\right)=0. \label{propmu}
	\end{equation}
	Taking $\mu(n)=\left(2n\right)^{2m}$, we get
	$$
	\lim_{n\rightarrow\infty}(-1)^{n}\mu\left(n+\frac{1}{2}\right)=\lim_{n\rightarrow\infty}(-1)^{n}\left(2n+1\right)^{2m}=0
	$$
	and 
	$$
		\lim_{n\rightarrow\infty}(-1)^{n}(2n+1)^{k}=0, \qquad k=2m.
	$$
	\item Let $k$ is odd, $k=2m-1$. \newline
	From Theorem \ref{theor:gen_inf_sum} for $\epsilon=-1$ and $t=1$, we have (\ref{eq:1_theor_gen_inf_sum}), that is
	\begin{equation}
	\sum_{u=1}^{\infty}f(u)=-\frac{1}{2}\lim_{n\rightarrow\infty}f(n+1).   
	\end{equation}
	The function $f(n)=(-1)^{n}\left(2n-1\right)^{2m-1}$ satisfies the condition of Theorem \ref{theor:gen_inf_sum}. Hence we obtain 
	\begin{equation}
	\sum_{u=1}^{\infty}(-1)^{u-1}\left(2u-1\right)^{2m-1}=-\frac{1}{2}\lim_{n\rightarrow\infty}(-1)^{n}\left(2n+1\right)^{2m-1}.   \label{sumlim}
	\end{equation}
	But according to (\ref{eq:from_diver_sin}) the sum at the left-hand side of (\ref{sumlim}) is equal to $0$. 
	Therefore, we get
	$$
	\sum_{u=1}^{\infty}(-1)^{u-1}\left(2u-1\right)^{2m-1}=-\frac{1}{2}\lim_{n\rightarrow\infty}(-1)^{n}\left(2n+1\right)^{2m-1}=0
	$$
	and finally
	$$
	\lim_{n\rightarrow\infty}(-1)^{n}\left(2n+1\right)^{k}=0, \qquad  k=2m-1.
	$$
\end{enumerate}
The proposition is proved completely.
\end{proof}

Now we are prepared to prove the theorem \ref{theor:polynom1}.

\begin{proof}[Proof of Theorem \ref{theor:polynom1}]
Because the limit of algebraic sum of finite number of sequences equals to
the algebraic sum of limits of  sequences 
$$
\lim_{n\rightarrow\infty}\sum_{u=1}^{m}\alpha_{u}F_{u}(n)=\sum_{u=1}^{m}\alpha_{u}\lim_{n\rightarrow\infty}F_{u}(n),
$$
where $\alpha_{u}$ are real numbers, we have
\begin{eqnarray}
\lim_{n\rightarrow\infty}(-1)^{n}f(n)&=&\lim_{n\rightarrow\infty}(-1)^{n}\left(a_{k}n^{k}+a_{k-1}n^{k-1}+...+a_{1}n+a_{0}\right)= \nonumber \\
&=& 
a_{k}\lim_{n\rightarrow\infty}(-1)^{n}n^{k}+a_{k-1}\lim_{n\rightarrow\infty}(-1)^{n}n^{k-1}+ \ldots +\nonumber \\
&&{}
+a_{1}\lim_{n\rightarrow\infty}(-1)^{n}n + a_{0}\lim_{n\rightarrow\infty}(-1)^{n}. \nonumber
\end{eqnarray}

To prove the theorem, it is sufficient to show that for every non-negative integer value $\sigma$
\begin{equation}
\lim_{n\rightarrow\infty}(-1)^{n}n^{\sigma}=0. \label{sigma}
\end{equation}
The proof is by induction over $\sigma$. 
\begin{enumerate}
	\item Let $\sigma=0$. From the formula (\ref{propmu}), putting $\mu\left(n+\frac{1}{2}\right) \equiv 1$, we immediately obtain that formula (\ref{sigma}) holds true.
	\item Assume now that (\ref{sigma}) holds for all positive integers less than some natural number $k$, $\sigma<k$. 
	Then, from (\ref{limzero}), applying the binomial theorem to $(2n+1)^{k}$, we get the expansion
	\begin{multline*}
	2^{k}\lim_{n\rightarrow\infty}(-1)^{n}n^{k} + 2^{k-1}\binom{k}{1}\lim_{n\rightarrow\infty}(-1)^{n}n^{k-1} +  \\
	2^{k-2}\binom{k}{2}\lim_{n\rightarrow\infty}(-1)^{n}n^{k-2} + \ldots + \lim_{n\rightarrow\infty}(-1)^{n}=0,  
	\end{multline*}
	in which, by the inductive assumption, all terms except for the first are equal to zero. Therefore, the first term will be equal to zero as well
	$$
	2^{k}\lim_{n\rightarrow\infty}(-1)^{n}n^{k}=0
	$$
	and finally
	$$
	\lim_{n\rightarrow\infty}(-1)^{n}n^{k}=0.
	$$
	So, the formula (\ref{sigma}) is also true for $\sigma=k$.
\end{enumerate}		
	Thus, by virtue of the induction, we obtain that the formula (\ref{sigma}) holds for all non-negative integer values $\sigma$. The theorem is proved.
\end{proof}

\begin{theorem} \label{theor:polynom2}
Suppose $f(x)$ is a polynomial defined over the field of real numbers, $x\in \mathbb{R}$. Then, we have
\begin{equation}
\lim_{n\rightarrow\infty}f(n)=\int\limits_{-1}^{0}f(x)dx. \label{eq:theor_polynom2}
\end{equation}
\end{theorem}
\begin{proof}
To prove the theorem, in analogy to the proof of Theorem \ref{theor:polynom1}, it is sufficient to show that the relation
\begin{equation}
\lim_{n\rightarrow\infty}n^{\sigma}=\frac{(-1)^{\sigma}}{\sigma+1} \label{sigma_fraction}
\end{equation}
holds for every non-negative integer value $\sigma$.

We give here a sketch of the proof.

At first step, we have the following equality
$$
\sum_{u=1}^{n}u^{2k}+\sum_{u=1}^{n}(-1)^{u-1}u^{2k}=2\sum_{u=1}^{[n/2]}(2u-1)^{2k}+\left(n^{2k}-(-1)^{n}n^{2k}\right), \quad \forall \, n, k.
$$
Then taking the limit and relying on axioms \textbf{A1} and \textbf{A2} and formula (\ref{sigma}), we obtain the relation
$$
\sum_{u=1}^{\infty}u^{2k}+\sum_{u=1}^{\infty}(-1)^{u-1}u^{2k}=2\sum_{u=1}^{\infty}(2u-1)^{2k}+\lim_{n\rightarrow\infty}n^{2k}.
$$
	The function $f(n)=\left(2n-1\right)^{2k}$ satisfies the condition of Theorem \ref{theor:gen_inf_sum}. Hence, in view of (\ref{eq:1_theor_gen_inf_sum}), we obtain 
	\begin{equation}
	\sum_{u=1}^{\infty}\left(2u-1\right)^{2k}=-\frac{1}{2}\lim_{n\rightarrow\infty}\left(2n+1\right)^{2k}.    \label{sumlimit}
	\end{equation}
Now according to (\ref{eq:sum_even_powers}), (\ref{eq:sum_alt_even_powers}) and (\ref{sumlimit}), we get
\begin{equation}
\lim_{n\rightarrow\infty}(2n+1)^{2k}=\lim_{n\rightarrow\infty}n^{2k}.  \label{eq:lim_even_powers}
\end{equation}
Then, one shows that 
\begin{equation}
\sum_{u=0}^{2k}(-1)^{u}\: 2^{2k-u}\binom{2k}{u}\frac{1}{2k+1-u}=\frac{1}{2k+1} \label{eq:1_proof_polynom2}
\end{equation}
and
\begin{equation}
\sum_{u=0}^{2k-1}(-1)^{u}\: 2^{2k-1-u}\binom{2k-1}{u}\frac{1}{2k-u}=0. \label{eq:2_proof_polynom2}
\end{equation}
Relying on Proposition \ref{prop:odd_psi} and Lemma \ref{lemma2} (or formula (\ref{eq:2_cor_lemma2})), we obtain
$$
\lim_{n\rightarrow\infty}\psi\left(n+\frac{a}{2}\right)=-\lim_{n\rightarrow\infty}\psi\left(n+1-\frac{a}{2}\right) \qquad \forall \, a \in \mathbb{Z}
$$
and in particular
\begin{equation}
\lim_{n\rightarrow\infty}(2n+1)^{2k-1}=0, \qquad \forall \, k\in \mathbb{N}.  \label{eq:2_rem_odd_psi}
\end{equation}

Then, following the induction over $\sigma$, we assume that the equality (\ref{sigma_fraction}) holds for all non-negative integers less than some natural number $k$. Further, depending on the parity of $k$ one uses the formulas (\ref{eq:1_proof_polynom2}) and (\ref{eq:lim_even_powers}) or formulas (\ref{eq:2_proof_polynom2}) and (\ref{eq:2_rem_odd_psi}), and then proves that (\ref{sigma_fraction}) holds for $\sigma=k$. And since the formula (\ref{sigma_fraction}) holds for $\sigma=0$, by the induction hypothesis it holds for all non-negative integer values of $\sigma$.

This completes the proof.
\end{proof}


Using Theorems \ref{theor:polynom1} and \ref{theor:polynom2}, we establish some formulas for sums of infinite arithmetic series. 

\begin{theorem} \label{theor:arith}
Let $a_1$, $a_2$,...,$a_u$,... be an infinite arithmetic progression, i. e. $a_u=a_1+(u-1)d$ and $d\geq 0$. Then for an infinite arithmetic series we have
\begin{equation}
\sum_{u=1}^{\infty}a_u=\frac{5d-6a_1}{12}. \label{eq:theor_arith}
\end{equation}
\end{theorem}
\begin{proof}
The sum $S_n$ of the first $n$ terms of arithmetic progression with difference $d$ and first term $a_1$ is defined by 
\begin{equation}
S_n=\sum_{u=1}^{n}a_u=\frac{2a_1+(n-1)d}{2}\:n.    \label{arithmsum}
\end{equation}
Since $S_n$ is the polynomial, then passing to the limit and using (\ref{eq:theor_polynom2}), we get
$$
\lim_{n\rightarrow\infty}S_n = \lim_{n\rightarrow\infty}\sum_{u=1}^{n}a_u = \int\limits_{-1}^{0}\left(\frac{2a_1+(x-1)d}{2}\right)xdx = \frac{5d-6a_1}{12}
$$
and, in view of axiom \textbf{A1}, we finally obtain
$$
\lim_{n\rightarrow\infty}S_n = \sum_{u=1}^{\infty}a_u=\frac{5d-6a_1}{12}.
$$
The theorem is proved.
\end{proof}

\begin{example}
\begin{align*}
	&\sum_{u=1}^{\infty}1=1+1+1+...     =-\frac{1}{2},  && (d=0)   && \textup{\cite{titch}} \\ 
	&\sum_{u=1}^{\infty}u=1+2+3+...     =-\frac{1}{12}, && (d=1)   && \textup{\cite{titch}} \\
	&\sum_{u=1}^{\infty}(2u-1)=1+3+5+...=\frac{1}{3},   && (d=2)   && \textup{\cite{hardy}}
\end{align*}
\end{example}

\begin{theorem} \label{theor:alt_arith}
Let $a_1$, $a_2$,...,$a_u$,... be an infinite arithmetic progression, i. e. $a_u=a_1+(u-1)d$ and $d\geq 0$. Then for an alternating infinite arithmetic series we have
\begin{equation}
\sum_{u=1}^{\infty}(-1)^{u-1}\: a_u=\frac{2a_1-d}{4}. \label{eq:theor_alt_arith}
\end{equation}
\end{theorem}

To prove the theorem \ref{theor:alt_arith}, we need the following result which we establish
relying on Theorem \ref{theor:polynom1}.

\begin{theorem} \label{theor:alphabeta}
Let $\alpha(x)$ and $\beta(x)$ be elementary functions defined on $\mathbb{Z}$ which satisfy the condition $\alpha(x)-\beta(x)=f(x)$, where $f(x)$ is a polynomial. Suppose that $\mu(x)$ is a function such that $\mu(x)=\alpha(x)$ if \, $2\mid x$ and $\mu(x)=\beta(x)$ if \, $2\nmid x$. Then
$$
\lim_{n\rightarrow\infty}\mu(n)=\frac{1}{2}\lim_{n\rightarrow\infty}\left(\alpha(n)+\beta(n)\right).
$$
\end{theorem}
\begin{proof}[Proof of Theorem \ref{theor:alphabeta}]
The function $\mu(x)$ that satisfies the condition of the theorem can be represented as the sum of two functions
$$
\mu(n)=\mu_{1}(n)+(-1)^{n} \: \mu_{2}(n),
$$
where 
$$
\mu_{1}(n)=\frac{1}{2}\left(\alpha(n)+\beta(n)\right) ~\mbox{and}~
\mu_{2}(n)=\frac{1}{2}\left(\alpha(n)-\beta(n)\right).
$$
Therefore
$$
\lim_{n\rightarrow\infty}\mu(n)=\lim_{n\rightarrow\infty}\left(\mu_{1}(n)+(-1)^{n}\:\mu_{2}(n)\right).
$$
But we have
$$
\mu_{2}(n)=\frac{1}{2}\left(\alpha(n)-\beta(n)\right)=\frac{1}{2}f(n).
$$
Hence
$$
\lim_{n\rightarrow\infty}\mu(n)=\frac{1}{2}\lim_{n\rightarrow\infty}\left(\alpha(n)+\beta(n)\right)+\frac{1}{2}\lim_{n\rightarrow\infty}(-1)^{n}f(n)
$$
and in view of (\ref{eq:theor_polynom1})
$$
\lim_{n\rightarrow\infty}\mu(n)=\frac{1}{2}\lim_{n\rightarrow\infty}\left(\alpha(n)+\beta(n)\right).
$$
The proof is completed.
\end{proof}

\begin{remark} 
From the proof of Theorem \ref{theor:alphabeta}, one can see that the theorem actually holds for a considerably wider class of elementary functions $\alpha(x)$ and $\beta(x)$, namely, for the functions which satisfy the condition
$$
\lim_{n\rightarrow\infty}(-1)^{n}\left(\alpha(n)-\beta(n)\right)=0.
$$
\end{remark}

\begin{proof}[Proof of Theorem \ref{theor:alt_arith}]
We give two alternative proofs of the theorem.
\begin{enumerate}
\item
Let
$$
\overline{S}_n=\sum_{u=1}^{n}(-1)^{u-1} \: a_u, \qquad \forall \, n.
$$
If $n=2n_1$, then
$$
\overline{S}_n=\sum_{u=1}^{n_1} a^{\prime}_u - \sum_{u=1}^{n_1} b^{\prime}_u,
$$
where $a^{\prime}_u=a_1+(u-1)2d$ and $b^{\prime}_u=a_1+d+(u-1)2d$.\\
Therefore
$$
\overline{S}_n=\left(\frac{2a_1+(n_1-1)2d}{2}\right)n_1 - \left(\frac{2a_1+2d+(n_1-1)2d}{2}\right)n_1 = -d n_1
$$
and
$$
\overline{S}_n=-d \: \frac{n}{2}, \qquad (2\mid x).
$$
If $n=2n_1+1$, then
$$
\overline{S}_n=\sum_{u=1}^{n_1+1}a^{\prime}_u - \sum_{u=1}^{n_1}b^{\prime}_u = a_1+dn_1
$$
and
$$
\overline{S}_n=a_1 + d \: \frac{n-1}{2}, \qquad (2\nmid x).
$$
So, the function $\overline{S}_n$ takes the values of $-d(n/2)$ if $2\mid x$ and of $a_1+d(n-1)/2$ if $2\nmid x$. 
Hence, in view of Theorem \ref{theor:alphabeta}, since
$$
\left(-d\:\frac{n}{2}\right)-\left(a_1+d\:\frac{n-1}{2}\right)=-dn-\frac{2a_1-d}{2}  
$$
is the polynomial of degree $n$, we have
$$
\lim_{n\rightarrow\infty}\overline{S}_n=\lim_{n\rightarrow\infty}\sum_{u=1}^{n}(-1)^{u-1}a_u=\frac{1}{2}\lim_{n\rightarrow\infty}\left(a_1+\frac{d(n-1)}{2}-\frac{dn}{2}\right)=\frac{2a_1-d}{4}
$$
and finally
$$
\sum_{u=1}^{\infty}(-1)^{u-1}a_u=\frac{2a_1-d}{4}.
$$

\item
Let us take the equality
$$
\overline{S}_n + S_n = 2\sum_{u=1}^{[n/2]}a_{2u-1}+\bigl(1-(-1)^{n}\bigr)a_n,
$$
where $a_{2u-1}=a_1+(u-1)2d$, $a_n=a_1+(n-1)d$, and $S_n$ is defined by (\ref{arithmsum}).
Then, passing to the limit and taking into account (\ref{sigma}), 
Theorem \ref{theor:arith}, axiom \textbf{A2}, and formula $\lim_{n\to\infty}n=-1/2$ which follows from (\ref{eq:2_rem_odd_psi}),  
we get
$$
\lim_{n\rightarrow\infty}\overline{S}_n + \lim_{n\rightarrow\infty}S_n = 2\sum_{u=1}^{\infty}a_{2u-1}+\lim_{n\rightarrow\infty}a_n,
$$
that is we have
$$
\sum_{u=1}^{\infty}(-1)^{u-1}a_u + \frac{5d-6a_1}{12} = 2\biggl(\frac{10d-6a_1}{12}\biggr)+a_1-\frac{3}{2}d
$$
and
$$
\sum_{u=1}^{\infty}(-1)^{u-1}a_u = \frac{2a_1-d}{4}.
$$
\end{enumerate}
The theorem is proved.
\end{proof}

\begin{example}
\begin{align*}
	&\sum_{u=1}^{\infty}(-1)^{u-1}=1-1+1-1+...      =\frac{1}{2},  && (d=0)  && \textup{\cite{hardy}} \\
	&\sum_{u=1}^{\infty}(-1)^{u-1}u=1-2+3-4+...     =\frac{1}{4},  && (d=1)  && \textup{\cite{hardy}} \\
	&\sum_{u=1}^{\infty}(-1)^{u-1}(2u-1)=1-3+5-7+...=0,            && (d=2)  && \textup{\cite{hardy}}
\end{align*}
\end{example}

\begin{remark}
In the above theorem \ref{theor:alt_arith}, instead of using axiom \textbf{A2} in the second proof, one can rely on the following fact. 
\begin{proposition}
For any polynomial $F(n)$ we have
$$
\lim_{n\to\infty}F(n)=\lim_{n\to\infty}F\left(\frac{2n-1+(-1)^n}{4}\right).
$$
\end{proposition}
In particular case, if $F(n)=\sum_{u=1}^{\left[n/2\right]}f(u), \, \forall \, n$, 
then $\lim_{n\to\infty}F(n)=\sum_{u=1}^{\infty}f(u)$.
\end{remark}

\begin{theorem} \label{theor:geom}
Let $g\neq 1$ be a real number, $g\in\mathbb{R}$. Then for an infinite geometric series we have
\begin{equation}
\sum_{u=0}^{\infty}g^u = \frac{1}{1-g}. \label{eq:theor_geom}
\end{equation}
\end{theorem}
\begin{proof}

For $|g|<1$ the validity of the theorem is obvious, since (\ref{eq:theor_geom}) reduces to the sum formula for infinite geometric series. 

Thus, let $|g|>1$. Then $\left|\dfrac{1}{g}\right|<1$ and
$$
\sum_{u=1}^{\infty}\frac{1}{g^u}=\frac{1}{g-1}.
$$

It is clear that the function $\dfrac{1}{g^x}$ is regular. Hence, in view of (\ref{eq:2_from_lemma3}), we get
$$
\sum_{u=1}^{\infty}\frac{1}{g^u}=-\sum_{u=0}^{\infty}g^u
$$
and
$$
\sum_{u=0}^{\infty}g^u=\frac{1}{1-g}.
$$

The case $g=-1$ reduces to the Theorem \ref{theor:alt_arith} with $a_1=1$ and $d=0$. So, it follows that
for $g=-1$ the sum is also equal to $1/(1-g)$.

This completes the proof.
\end{proof}

\begin{example}
\begin{align}
	&\sum_{u=0}^{\infty}2^u = 1+2+4+...          =-1           && \textup{\cite{hardy}} \\
	&\sum_{u=0}^{\infty}(-1)^{u}\:2^u = 1-2+4-...=\frac{1}{3}  && \textup{\cite{hardy,kline}} 
\end{align}
\end{example}



\section{Extension of the class of regular functions}\label{sec:exten} 

In Section \ref{sec2} we defined regular functions $f(x)$ by 

\begin{equation}
F(z+1)-F(z)=f(z), \quad \forall \, z\in \mathbb{Z}    \label{eq:regucond}
\end{equation}
with the condition that a primitive function $F(x)$ is elementary.

In order to extend the class of regular functions by constructing them from non-elementary functions, we have to redefine the notion of regularity and the notion of odd/even function. To do this, it suffices to refine the definitions by means of limit conditions.

Before we proceed, it should be noted that for functions $f(x)$ and $F(x)$ connected by (\ref{eq:regucond}) the limiting relation 
\begin{equation}
\lim_{n\rightarrow\infty}F(n+1)-\lim_{n\rightarrow\infty}F(n)=\lim_{n\rightarrow\infty}f(n)  \label{eq:prop_exten_main}
\end{equation}
always holds in case the primitive function $F(x)$ is elementary.

So, we have the following

\begin{proposition} \label{prop:exten}
For any two functions $f(x)$ and $F(x)$ connected by the relation
$F(z+1)-F(z)=f(z), \: \forall \, z\in \mathbb{Z}$, where $F(x)$ is elementary, we have
$$
\lim_{n\rightarrow\infty}F(n+1)-\lim_{n\rightarrow\infty}F(n)=\lim_{n\rightarrow\infty}f(n). 
$$
\end{proposition}

\begin{proof}
This proposition can be proved by different ways.
We only present  two proofs of this statement.
\begin{itemize}
\item[(I)] Using formula (\ref{regeq}), we have
\begin{equation}
F(n+1)-F(1)=\sum_{u=1}^{n}f(u), \quad \forall \, n  \label{eq:prop_exten_1}
\end{equation}
and
\begin{equation}
F(n)-F(1)=\sum_{u=1}^{n}f(u)-f(n), \quad \forall n. \label{eq:prop_exten_2}
\end{equation}
Subtracting (\ref{eq:prop_exten_2}) from (\ref{eq:prop_exten_1}), we get
$$
S_n = F(n+1)-F(n)-f(n)=\sum_{u=1}^{n}\omega(u), \quad \forall \, n
$$
where $\omega(u)\equiv 0$.
Now passing to the limit, we have
$$
\lim_{n\rightarrow\infty}S_n=\lim_{n\rightarrow\infty}\left(F(n+1) - F(n) - f(n)\right)=\sum_{u=1}^{\infty}\omega(u).
$$
Due to our claim that every elementary function defined on $\mathbb{Z}$ has a certain limit, we can write
$$
\lim_{n\rightarrow\infty}F(n+1)-\lim_{n\rightarrow\infty}F(n) - \lim_{n\rightarrow\infty}f(n)=\sum_{u=1}^{\infty}\omega(u).
$$
And since $\sum_{u=1}^{\infty}\omega(u)=0$, we finally obtain
$$
\lim_{n\rightarrow\infty}F(n+1)-\lim_{n\rightarrow\infty}F(n)=\lim_{n\rightarrow\infty}f(n),
$$
which completes the proof. 
\item[]
\item[(II)] According to the formula (\ref{regeq}) and formula (\ref{eq:1_from_lemma2}), we have
\begin{equation}
F(-n)-F(1)=\sum_{u=1}^{-(n+1)}f(u)=-\sum_{u=0}^{n}f(-u), \quad \forall \, n  \label{eq:prop_exten_3}
\end{equation}
and
\begin{equation}
F(n)-F(1)=\sum_{u=1}^{n}f(u)-f(n), \quad \forall \, n. \label{eq:prop_exten_4}
\end{equation}
Subtracting (\ref{eq:prop_exten_4}) from (\ref{eq:prop_exten_3}), we get
$$
F(-n)-F(n)-f(n)=-\left(\sum_{u=0}^{n}f(-u)+\sum_{u=1}^{n}f(u)\right), \quad \forall \, n.
$$
Passing to the limit, and taking into account that every elementary function defined on $\mathbb{Z}$ has a certain limit, and also using Lemma \ref{lemma2}, we obtain
$$
\lim_{n\rightarrow\infty}F(n+1)-\lim_{n\rightarrow\infty}F(n) - \lim_{n\rightarrow\infty}f(n)=
-\left(\sum_{u=1}^{\infty}f(-u)+\sum_{u=0}^{\infty}f(u)\right),
$$
which in view of (\ref{eq:2_from_lemma3}) finally gives
$$
\lim_{n\rightarrow\infty}F(n+1)-\lim_{n\rightarrow\infty}F(n)=\lim_{n\rightarrow\infty}f(n).
$$
\end{itemize}

The proposition is proved.
\end{proof}

\begin{example}
Consider the primitive function
$$ F(n) =              
									\begin{cases}                   
									1, & \quad n = 0,\pm 2,\pm 4,\dots , \\                   
									0, & \quad n = \pm 1,\pm 3,\dots              
									\end{cases}       
						\; = \; \frac{(-1)^n+1}{2}
$$
which generates the regular function
$$ f(n) =              
									\begin{cases}                   
								  1, & \quad  n = \pm 1,\pm 3,\dots \\                   
								 -1 , & \quad  n = 0,\pm 2,\pm 4,\dots
									\end{cases}       
						 \; = \; (-1)^{n+1}
$$
since we have $F(x+1)-F(x)=f(x), \ \forall \, x \in\mathbb{Z}$. It is easy to see that for these functions 
the limiting relation (\ref{eq:prop_exten_main}) 
$$
\lim_{n\rightarrow\infty}F(n+1)-\lim_{n\rightarrow\infty}F(n) = \lim_{n\rightarrow\infty}f(n)
$$
does hold. From Lemma  \ref{lemma2} $\lim_{n\to\infty}F(n+1)=\lim_{n\to\infty}F(-n)$ and from the definition of $F(n)$ we have $F(-n)=F(n)$, hence
the left-hand side is equal to zero, and according to the formula (\ref{propmu}) $\lim_{n\to\infty}f(n)=0$.
The same follows immediately from the formula (\ref{propmu}) if we take all three limits separately.
\end{example}

However, the limiting relation (\ref{eq:prop_exten_main}), and also some of our statements, might not hold in case a primitive function $F(x)$ is non-elementary.
The following example demonstrates this.

\begin{example} \label{ex:nonelement}
Let us consider the function given as
$$ f(x) =               
									\begin{cases}                   
									0, & \quad x\neq 0, \\                   
									1, & \quad x=0              
									\end{cases}       						 
$$
the primitive function of which is
$$ F(x) =               
									\begin{cases}                   
									1, & \quad x = 1,2,3,\dots , \\                   
									0, & \quad x = 0,-1,-2,\dots              
									\end{cases}       					
$$
since $F(x+1)-F(x)=f(x), \ \forall \, x \in\mathbb{Z}$. 
The function $F(x)$ is not elementary. 
One can see that the relation (\ref{eq:prop_exten_main}) is not satisfied for the functions $F(x)$ and $f(x)$: 
within the framework of our  method, in view of Lemma  \ref{lemma2} we have $\lim_{n\to\infty}F(n+1)=\lim_{n\to\infty}F(-n)$, and therefore
$$
\lim_{n\rightarrow\infty}F(n+1)-\lim_{n\rightarrow\infty}F(n) = -1 \neq \lim_{n\rightarrow\infty}f(n).
$$
Moreover, the function $f(x)$ is even and according to the Theorem \ref{theorem:inf_sum} we would have had
$\sum_{u=1}^{\infty}f(u)=-f(0)/2=-1/2$, whereas in our case $\sum_{u=1}^{\infty}f(u)=0$. However, one should not think about inconsistencies here, 
and just have to notice that if we are still within the definition \ref{def:reg} then according to it the function $f(x)$ is not regular because the function $F(x)$ is non-elementary, and moreover, it does not satisfy the relation (\ref{eq:prop_exten_main}), which is highly essential. 

We return to this example after we refine our definitions below.
\end{example}

The new definition of regularity explicitly includes the limiting relation (\ref{eq:prop_exten_main}), and the notion of parity of functions has to be refined so that the classical conditions of function oddness and evenness should also hold at infinity. These definitions are now formulated as follows.

\begin{definition} \label{def:exten}
The function $f(x)$, $x\in \mathbb{Z}$, is called \textit{regular} if there exists a function $F(x)$, not necessarily elementary, such that 
$$
F(z+1)-F(z)=f(z), \quad \forall \, z\in \mathbb{Z}
$$
and
$$
\lim_{n\rightarrow\infty}F(n+1)-\lim_{n\rightarrow\infty}F(n)=\lim_{n\rightarrow\infty}f(n).
$$
\end{definition}

\begin{definition} \label{def:neweven}
The function $f(x)$, $x\in \mathbb{Z}$, is called \textit{even} if 
$$
f(-z)=f(z), \quad \forall \, z\in \mathbb{Z}
$$
and
$$
\lim_{n\rightarrow\infty}f(-n)=\lim_{n\rightarrow\infty}f(n).
$$
\end{definition}

\begin{definition} \label{def:newodd}
The function $f(x)$, $x\in \mathbb{Z}$, is called \textit{odd} if 
$$
f(-z)=-f(z), \quad \forall \, z\in \mathbb{Z}
$$
and
$$
\lim_{n\rightarrow\infty}f(-n)=-\lim_{n\rightarrow\infty}f(n).
$$
\end{definition}


Returning to the example \ref{ex:nonelement}, one can see that the function $f(x)$, according to the definition \ref{def:neweven}, is not even, since $\lim_{n\to\infty}f(-n)\neq\lim_{n\to\infty}f(n)$. And so it agrees with all the formulas in which the function is not required to be even, whereas it does not satisfy the theorem \ref{theorem:inf_sum}, where the function evenness is an essential requirement.

Thus, the class of regular functions can be substantially extended by using non-elementary functions,  
but one should be very careful when doing this.


\section{Discussion and concluding remarks}\label{sec:conc}

In this concluding section we propose to recall some of our main results,  and  briefly outline some directions which seem to be useful to develop.

In this work we presented the general and unified method of summation, which is regular and consistent with the classical definition of summation. 

We introduced the new ordering relation on the set of integer numbers and defined sum   in Section \ref{sec2}, along with the properties of summation over newly ordered number line presented in Section \ref{sec:sums}. As one can see, this is the foundation the entire theory rests on.

Our results of Section \ref{sec:main}
are for the present the main tools of our method. Those theorems are found useful in applications, for example, in finding sums of the certain classes of series which somehow involve regular functions. Besides, finding limits of unbounded and oscillating functions and sequences is another advantage, since the knowledge of certain limits has significance in series summations.

Some of the propositions, mainly those concerning the limits of unbounded and oscillating functions and sequences,
being true  within the framework of our theory, might not hold in the usual ``epsilon-delta''--definition of limit in classical analysis; this is also due to the fact that, as it is known, modern analysis does not possess general tools for finding the
limits of sequences and functions with oscillation or unboundedness; the matter is that the topology\footnote{but not the axioms, which are obviously consistent with the standard topology of $\mathbb{R}$} of the newly ordered number line may be different\footnote{topological and algebraic properties are the subjects of separate studies} from that in classical sense. Perhaps, lemmas \ref{lemma:1} and \ref{lemma2} may shed some light on the understanding of topological properties of the number line equipped with the introduced  ordering relation.

We would like to note that the developed theory can be extended to products.
The equality analogous to (\ref{def:first_sum_eq}) can also be defined for products $\prod_a^b$ with arbitrary limits $a$ and $b$, $(a ^>_< b)$. Namely
\begin{definition}
For any $a,b\in\mathbb{Z}$
$$
\prod_{u=a}^{b}f(u)=\prod_{u\in\mathbb{Z}_{a,b}}f(u).
$$
\end{definition}

We have the following
\begin{proposition} \label{prop:product}
Suppose $f(x)$ is a regular function defined on $\mathbb{Z}$ and let $f(x)$ is such that \, $\prod_{u=1}^{0}f(u)=1$. Then
\begin{equation}
\prod_{u=0}^{-n}f(u)=\left(\prod_{u=1}^{n-1}f(-u)\right)^{-1}.  \label{eq:product}
\end{equation}
\end{proposition}

So, we can expect that similar results and theory, as we elaborated for series in this paper, can be developed for products.
For example, relying on proposition \ref{prop:product}, the factorial function $n!$, initially defined only for natural numbers, can be easily expanded on the set of all integers. 
This can be done if one notices that the function $f(x)=x$ satisfies the proposition \ref{prop:product}, since for the factorial $\lambda(a)=\prod_{u=1}^{a}u=a!$ one formally has $\lambda(0)=\prod_{u=1}^{0}u=0!=1$; then putting $f(x)=x$ in (\ref{eq:product}) one can notice that analogously to gamma-function the function $\lambda(a)$ has simple poles at $a=-n$ with residue $(-1)^{n-1}/(n-1)!$; using this, one can define the binomial coefficients for negative numbers.

One of the advantages of our approach is that it provides a systematic method for summation of series in a certain class of functions discussed in the paper. Using the same techniques and proofs one can further develop, refine, or generalize the presented results.

For example, for the formulas (\ref{eq:theor_arith}) and (\ref{eq:theor_alt_arith}) for infinite (alternating) arithmetic series it would be interesting to study the question of whether it is possible to obtain these  formulas with use of analytic continuation techniques. 
It would be natural to study some other special types of sequences to get similar finite formulas as for arithmetic series.

The main restriction of the method is that it is discrete, that is, it works when the function is considered at integer points; 
for example, at the current level of development of the method we cannot directly evaluate $\zeta(1/2)$.
However, further development of the theory should make it work in the continuous case, thus extending it to real arguments. This problem does not seem to be straightforward. 

This work shows that there is a very natural way of extending
summations to the case when the upper limit of summation is less than the lower or even negative. 
Many classical results nicely fall into this more general theoretical setting, for example, 
the geometric series, the binomial theorem and the Gauss hypergeometric series \cite{bag-jmr}, some series involving Bernoulli numbers and polynomials \cite{bag-phan}, and Riemann's zeta and related functions \cite{bag-phan,bag-aip}. Thus, using our summation method, recently we evaluated the Riemann's zeta function and related zeta functions at integer points \cite{bag-jmr,bag-phan,bag-aip} in elementary fashion, for example, we have recovered the standard representation for $\zeta(2k)$ and $\zeta(-k)$. Perhaps, use of the method to study zeta values at odd positive integers, including studies on series representations, seems to be interesting. It would also be interesting to apply our method for some other Dirichlet series or classes of series associated with the zeta and related functions. 
 
Further, it would be a natural attempt to connect this research with a topic that might bear a relation to multiple zeta functions; the first topic might be evaluating and studying multiple zeta values, as well as Euler and harmonic sums. For example, in \cite{basu} T.~Apostol and A.~Basu use generalized telescoping sums to propose a new method to investigate Euler sums. 
So, it seems promising to employ the presented summation method, which uses, to an extent, a ``telescoping'' idea as well, to make an attempt to develop yet another approach to studying those sums.

The method developed here fundamentally differs from the others in its foundational aspects; though having some physical flavor, it has as an underlying premise an original viewpoint on the number line, which seems to be of independent interest.

It is expected that the developed summation theory will continue to yield new results, and the theory will appear to be useful in various applications.

\section*{Acknowledgements}
I  wish to express my sincere gratitude  to Harold Diamond (University of Illinois at Urbana-Champaign)
for reading the manuscript, 
constructive comments and  suggestions throughout the paper, and for his keen interest to this work.
I also would like to thank Robert Sczech (Rutgers University) for useful discussions and remarks, and his interest to the ideas developed in this paper.



\end{document}